\documentclass{amsproc}

\usepackage{amssymb}

\usepackage{graphicx}

\usepackage[cmtip,all]{xy}

\usepackage[all]{xy}
\usepackage{amsthm}
\usepackage{amscd}
\usepackage{amsmath}
\usepackage{amsfonts}
\usepackage{amssymb}
\usepackage{color}
\usepackage{stmaryrd}
\usepackage{graphicx}%
\usepackage{hyperref}
\hypersetup{
    bookmarks=true,         % show bookmarks bar?
    unicode=false,          % non-Latin characters in Acrobat��s bookmarks
    pdftoolbar=true,        % show Acrobat��s toolbar?
    pdfmenubar=true,        % show Acrobat��s menu?
    pdffitwindow=false,     % window fit to page when opened
    pdfstartview={FitH},    % fits the width of the page to the window
    pdftitle={My title},    % title
    pdfauthor={Author},     % author
    pdfsubject={Subject},   % subject of the document
    pdfcreator={Creator},   % creator of the document
    pdfproducer={Producer}, % producer of the document
    pdfkeywords={keyword1, key2, key3}, % list of keywords
    pdfnewwindow=true,      % links in new PDF window
    colorlinks=false,       % false: boxed links; true: colored links
    linkcolor=red,          % color of internal links (change box color with linkbordercolor)
    citecolor=green,        % color of links to bibliography
    filecolor=magenta,      % color of file links
    urlcolor=cyan           % color of external links
}

\newtheorem{theorem}{Theorem}[section]

\theoremstyle{definition}
\newtheorem{definition}[theorem]{Definition}
\newtheorem{example}[theorem]{Example}
\newtheorem{proposition}[theorem]{Proposition}

\theoremstyle{remark}
\newtheorem{remark}[theorem]{Remark}

\numberwithin{equation}{section}

\newcommand*{\emptycomment}[1]{}

\begin{document}

\title{Universal Finite Subgroup of the Tate Curve}


%    author one information

\author{Zhen Huan}

\address{Zhen Huan, Center for Mathematical Sciences,
Huazhong University of Science and Technology, Hubei 430074, China} \curraddr{}
\email{2019010151@hust.edu.cn}
\thanks{The author is supported by  the Young Scientists Fund of the National Natural Science Foundation of China (Grant No. 11901591) for the project  "Quasi-elliptic cohomology and its application in geometry and topology", and the research funding from Huazhong University of Science and Technology.}

%    author two information
%\author{}
%\address{}
%\curraddr{} \email{}
%\thanks{}

%\subjclass[2000]{Primary }
%    The 2010 edition of the Mathematics Subject Classification is
%    now available.  If you are citing a classification from the
%    new scheme, use the following input coding instead.
\subjclass[2010]{Primary 55}

\date{}

\begin{abstract}

In \cite{KM85} Katz and Mazur discuss the moduli problem of the subgroup-schemes of elliptic curves. We give the classification of the finite subgroups of the Tate curve in \cite{HuanPower}. Moreover, Katz and Mazur define the universal finite subgroup of an elliptic curve.
In this paper we give an explicit construction of the universal finite subgroup of the Tate
curve via isogenies  and the stringy power operation of Tate K-theory.
\end{abstract}

\maketitle


\section{Introduction}

An elliptic curve $E\longrightarrow S$ over a base $S$ is an smooth and proper abelian group $S$-scheme $E$ whose fiber at every geometric point is an elliptic curve. In \cite[Chapter 6]{KM85}
Katz and Mazur discussed the moduli problem [N-Isog], where the symbol [N-Isog] means isogenies with finite kernel of order $N$. The $S$-scheme [N-Isog]($E/S$) is defined to be the set of finite locally free commutative
$S$-subgroup-schemes $G<E[N]$ which are of rank $N$ over $S$. In other words, [N-Isog]($E/S$) is the set of subgroup schemes of rank $N$ in $E$.
In \cite[Proposition 6.5.1, Theorem 6.8.1]{KM85} Katz and Mazur prove that this moduli problem is relatively representable and is finite and flat over (Ell).

Let $E$ be an elliptic curve over a commutative ring $A$. Among the subgroup-schemes of rank $N$ in $E[N]$, based on the idea in \cite{KM85}, we define  the universal finite subgroup of $E$. Fix an integer $N\geq 0$. There exists a ring homomorphism $A\longrightarrow B$ and a subgroup
$G < E_B$ of rank $N$ in $E_B$, where $E_B$ is the pullback of $E$ over $\mbox{Spec} \; (B)$,    with the following property.  The pair $(B, G<E_B)$ is the universal subgroup of rank $N$ of $E\longrightarrow \mbox{Spec}\;(A)$
in the sense that: given a ring homomorphism $A\longrightarrow C$ and a subgroup $H$ of rank $N$ in $E_C$, there exists a unique ring homomorphism $g: B\longrightarrow C$
compatible with the maps from $A$ such that $H=g^*(G)$,   i.e.    there is a pullback diagram in schemes of the form

\begin{equation}\xymatrix{H \ar@{>->}[r] \ar[d]&E_C \ar[d] \ar[r] &\mbox{Spec}\;(C) \ar[d] \\
G\ar@{>->}[r] &E_B\ar[r] \ar[d] &\mbox{Spec}\;(B)\ar[d] \\ & E\ar[r] &\mbox{Spec}\;(A).}\label{universaldef}\end{equation}

In this paper we show the universal finite subgroup of the Tate curve exists and provide an explicit construction of it.

 The Tate curve $Tate(q)$ is an  elliptic curve
over Spec $\mathbb{Z}((q))$. The formal group of it is the multiplicative group. Thus, it is a form of $K$-theory. We call the theory $K_{Tate}$. It is the elliptic cohomology theory that is first discovered; See \cite{Mor89}. 
In \cite{HuanPower} we discuss the moduli problem [N-Isog] for the Tate curve and prove in \cite[Theorem 7.4]{HuanPower}  that the finite subgroups of order $N$ of the Tate curve $Tate(q)$ can be classified by the Tate K-theory of the symmetric group $\Sigma_N$ modulo a certain transfer ideal. 
The application of an intermediate theory, quasi-elliptic cohomology, is essential in the
proof of this classification theorem, as shown in Section 6.3 in \cite{HuanPower}.

In this paper we construct the universal finite subgroup of order $N$ of the Tate curve $Tate(q)$. 
The finite subgroups of the Tate curve are the kernels
of isogenies. In light of this fact, we construct the universal finite subgroup of $Tate(q)$ as the kernel of an isogeny. Furthermore, as shown in 
\cite{AndoIsog}, there is some correspondence between  isogenies and operations among complex oriented cohomology theories under some condition.
By \cite[Theorem A, Section 6.3]{AndoIsog}, the isogenies of certain form give rise to operations of Tate K-theory with proper coefficient rings. The construction of the isogenies will be recalled in Section \ref{tatecurve}.  Especially, these operations have relations with the additive power operation of Tate K-theory, which is constructed from the stringy power operation of Tate K-theory \cite{Gan07}.

In \cite{HuanPower} we present a construction of the stringy power operation. We first construct a power operation $\{\mathbb{P}_N\}_N$ of quasi-elliptic cohomology, which can be viewed as the Tate K-theory with coefficient ring $\mathbb{Z}[q^{\pm}]$ and is introduced in Section \ref{qelldef}. This power operation is  related to the level
structure of the Tate curve. Its construction   mixes  power operations in
K-theory with the natural operations of dilating and rotating loops, and can be generalized to other equivariant cohomology theories.
 $\{\mathbb{P}_N\}_N$ can extend uniquely to the stringy power operation of Tate K-theory.
Moreover, from the power operation, we construct in  \cite[Proposition 6.5]{HuanPower}
the additive power operation $$\overline{P}_N: QEll(X/\!\!/G)\longrightarrow QEll(X/\!\!/G)\otimes_{\mathbb{Z}[q^{\pm}]} QEll(\mbox{pt}/\!\!/{\Sigma_N})/\mathcal{I}^{\Sigma_N}_{tr},$$ which is a ring homomorphism.
%which  sends $q$ to $q'$.
In \cite{AHS04} Ando, Hopkins and Strickland discuss the additive power operation
of Morava E-theories $$E^0\longrightarrow E^0(B\Sigma_{p^k})/I_{tr}.$$ Applying
Strickland's theorem in \cite{Str98} they show that it has a nice algebro-geometric interpretation in terms of the formal group and it takes the quotient by the universal
subgroup. The additive operation $\overline{P}_N$ plays an essential part in
the construction of the isogeny whose kernel is the universal
subgroup of order $N$ of the Tate curve.

In Section \ref{maincc}, we  prove the main theorem that the finite subgroup constructed in Section \ref{mainpart} is indeed the universal 
finite subgroup of the Tate curve..

\begin{theorem} The  universal finite subgroup of order $N$ of $Tate(q)$ is the pair $$G_{univ}:=(D_N, \mbox{Ker} (\psi)< i_N^*Tate(q)[N]),$$
where the $\mathbb{Z}((q))-$algebra $D_N$ and the isogeny $\psi$ are defined in Section \ref{mainpart},
in the sense that for any
$\mathbb{Z}((q))$-algebra $R$, there is a 1-1 correspondence which is natural
\begin{equation}\xymatrix{\{\mathbb{Z}((q))-\mbox{algebra maps }D_N\longrightarrow R
\}\ar@{^{<}-_{>}}[d]\\  \{\mbox{finite subgroup schemes }
G\leqslant Tate(q)[N]_R \mbox{   of order }N\}}%\label{11main}
\end{equation} where $Tate(q)[N]_R$ is the
pullback
$$\xymatrix{Tate(q)[N]_R\ar[r]\ar[d] &Tate(q)[N]\ar[d] \\ \mbox{Spec}\; (R)\ar[r]
&\mbox{Spec}\;(\mathbb{Z}((q )) )}.$$ It is a group scheme over
$\mbox{Spec}\;( R)$. In other words, given a subgroup scheme
$G\leqslant Tate(q)[N]_R$ of degree $N$, there exists a unique pullback
square
\begin{equation}\xymatrix{G \ar@{^{(}->}[r]\ar@{^{(}->}[d] &\mbox{Ker} (\psi) \ar@{^{(}->}[d] \\ Tate(q)[N]_R\ar[r]\ar[d] &i_N^*Tate(q)[N] \ar[d] \\
\mbox{Spec} \; (R)\ar[r] &\mbox{Spec}\; (D_N)
.}%\label{universaldef}
\end{equation} \label{mainthm0}
\end{theorem}

\bigskip

In Section \ref{tatecurve}, we recall the Tate curve and its torsion points. 
%Each finite subgroup of order $N$ of $Tate(q)$ is contained in the set of $N$-torsion points of $Tate(q)$, which  can be described explicitly.
%In Section 8.7.1 of \cite{KM85} Katz and Mazur discuss the torsion points of order $N$ of $Tate(q)$, 
%which can be classified by the Tate K-theory of the cyclic group $\mathbb{Z}/N\mathbb{Z}$; See \cite{Gan07}.
We discuss the finite subgroups of the Tate curve and their relation with isogenies. In Section
\ref{qellshort}, we recall Tate K-theory. Some constructions are made via quasi-elliptic cohomology.  
We also recall the classification theorems of  the Tate curve and the power operations of the theories
 that we need later in Section \ref{mainpart}.
The main references are \cite{HuanPower} and
\cite{Rez11}. 
In  Section \ref{mainpart} we construct the universal finite subgroup of $Tate(q)$ explicitly as the kernel of an isogeny constructed from the additive power operation. In Section \ref{maincc}  we show that the construction in Section \ref{mainpart} is indeed the universal finite subgroup and prove the main theorem.

\bigskip

\textbf{Acknowledgement} I would like to thank Charles Rezk. This project is motivated and directed by him. We had many helpful
conversations. I would also like to thank Matthew Ando and Nathaniel Stapleton for several discussions and helpful remarks.

\section{The Tate curve and its finite subgroups}\label{tatecurve}

In this section we introduce the Tate curve and its torsion points. We discuss the isogenies between the Tate curve in Example \ref{excisog}, 
\ref{fisogtate} and \ref{keyisogtate}, and the finite subgroups of it as kernel of isogenies.
 The main references are   \cite[Section 2.6]{AHS} and
  \cite[Section 8.7- 8.8]{KM85}.

The Tate curve  $Tate(q)$  is the elliptic curve  \begin{equation}E_q:
y^2+xy=x^3+a_4x+a_6 \label{tatedef1}\end{equation}  whose coefficients are given by the formal
power series in $\mathbb{Z}((q))$:\begin{equation} a_4(q)=-5\sum_{n\geqslant 1} n^3q^n/(1-q^n)   \,\,\,\,\,\,\,\,\,\,\,\,\,\,\,   a_6(q) =-\frac{1}{12}\sum_{n\geqslant 1}(7n^5+5n^3)q^n/(1-q^n). \label{tatedef2}\end{equation}

The Tate curve $Tate(q)$,  defined by \eqref{tatedef1} and  \eqref{tatedef2}, is a pointed curve of genus 1 over $\mathbb{Z}((q))$. 
%The point is located in the smooth part and 
It is a one-dimensional abelian group.
There is an isomorphism \begin{equation}
    \widehat{\mathbb{G}_m}\cong \widehat{Tate(q)}
\end{equation} of formal groups; see \cite[VII, 1.16]{DR73}. 
Since its formal group is multiplicative, the cohomology theory associated to the Tate curve is a form of $K$-theory; we call it Tate K-theory $K_{Tate}$. 

It is shown in \cite[Theorem 3.1 (b)]{Silverman1} that the Weierstrass cubic \eqref{tatedef1} has discriminant 
\begin{equation}
    \Delta(q)=q\prod_{n\geq 1} (1-q^n)^{24}.
\end{equation} This implies that over $\mathbb{Z}((q))$ the Tate curve is an elliptic curve. It is modeled on the multiplicative parameterization of elliptic curves over $\mathbb{C}$: for any complex number $q$ with $0< |q| <1$, $\mathbb{C}^*/q^{\mathbb{Z}}$ 
is an elliptic curve which fits into the exact sequence 
\begin{equation}
    \xymatrix{1\ar[r] &q^{\mathbb{Z}} \ar[r] &\mathbb{C}^* \ar[r] &\mathbb{C}^*/q^{\mathbb{Z}} \ar[r] &1}.
\end{equation}
Thus, we obtain a family of elliptic curves \[\mathbb{C}^*/q^{\mathbb{Z}}\rightarrow D^*(\mathbb{C})\] where $D^*(\mathbb{C})$ is the punctured open unit disk
\[D^*(\mathbb{C})=\{q\in\mathbb{C} \;  | \;  0< |q| < 1\}.\]
For $q\in D^*(\mathbb{C})$, $a_4(q)$ and $a_6(q)$ converge to complex numbers. Let $E_q$ denote the resulting elliptic curve over $\mathbb{C}$. 
As shown in \cite[p. 410]{Silverman1}, we have an analytic isomorphism \[\mathbb{C}^*/q^{\mathbb{Z}} \cong E_q.\]

If $F$ is a
non-archimedean field, let 
\[D^*(F)=\{q\in F \;  | \;  0< |q| < 1\}.\] We have the isomorphism of sets
\[\hom^{cts}(\mathbb{Z}((q)), F) \rightarrow D^*(F), \quad g \mapsto g(q).\] 
For any continuous homomorphism \[\mathbb{Z}((q))\xrightarrow{g} F,\] there is an isomorphism
of groups \[F^*/g(q)^{\mathbb{Z}} \cong g^* Tate (F).\]
See \cite[p. 423]{Silverman1} for more details.

\begin{example} \label{excisog}
%To give a description of $O_{Sub_N}$, first 

In this example we describe the finite subgroups and the 
isogenies for $E_q$, the analytic Tate curve over $\mathbb{C}$.
%The answer is already known, as shown below.

Let $N$ be any positive integer. To give a subgroup of  order $N$ of $E_q$, we pick a pair of integers
$(d, e)$ such that $N=de$ and
$d, e\geq 1$ and let $q'$ be a nonzero complex number such that
$q^d=q'^e$. Consider the isogeny
\begin{align}\psi_d: \mathbb{C}^*/q^{\mathbb{Z}}&\longrightarrow \mathbb{C}^*/q'^{\mathbb{Z}} \label{psidisg}\\x&\mapsto x^d. \end{align}
It is well-defined since $\psi_d(q^{\mathbb{Z}})\subseteq
q'^{\mathbb{Z}}$.

We can check that Ker$(\psi_d)$ has order $N$. Explicitly, it is
$$\{\mu_d^nq^{\frac{m}{e}}q^{\mathbb{Z}}  \;  |  \; n, m\in\mathbb{Z}\}$$ where $\mu_d$ is a $d$'th primitive root of 1 and $q^{\frac{1}{e}}$ is an $e$'th primitive root of $q$.
In fact $$\{\mbox{Ker}(\psi_d)   \;   | \; d \mbox{  divides
}N\mbox{  and  }d\geq 1\}$$ gives all the subgroups of
$\mathbb{C}^*/q^{\mathbb{Z}}$ of order $N$.

\end{example}

\begin{example} \label{fisogtate}
It is shown in \cite[Exercise 5.10]{Silverman1} a generalization of Example \ref{excisog}. 

For a $p-$adic field $K$, if $q, q'\in K$, $0<|q|,|q'| <1$ , and $q^d=q'^e$, then the function \[\overline{K}^*/q^{\mathbb{Z}} \rightarrow \overline{K}^*/q'^{\mathbb{Z}}, \quad u\mapsto u^d\]  lifts to an isogeny $E_q\rightarrow E_{q'}$ of elliptic curves over $K$, where $E_q$ and $E_{q'}$ are defined by the Tate curve equations
\eqref{tatedef1} \eqref{tatedef2}.

%\end{example} \begin{example}
\bigskip

Moreover, as shown in 
\cite[Section 6.3]{AndoIsog}, if $F$ is a non-archimedean field, there is an isogeny \[F^* /q^{\mathbb{Z}} \rightarrow 
F^* /q^{\mathbb{Z}}, \quad z\mapsto z^d\] for any positive integer $d$.

\end{example}
We recall a model for the torsion points of the Tate curve from \cite[Section 8.7]{KM85} and \cite[Section 2.3]{AndoIsog}. The $N$-torsion points
$T[N]$ of it is the disjoint union of $N$ schemes \[T_0[N], \;
\cdots \; T_{N-1}[N],\] where, for each integer $0\leq i< N$,
\begin{equation} T_i[N]=\mbox{Spec} \;(\mathbb{Z}[q^{\pm}][x]/(x^N-q^i)).\label{tindef}\end{equation} It fits
into a short exact sequence  of group schemes  \cite[(8.7.1.4)]{KM85} \begin{equation}\xymatrix{0\ar[r]
&\mu_N\cong T_0(N)  \ar[r]^>>>>>>{a_N} %@{^{(}->}[r]^>>>>>{a_N}
&T[N]\ar[r]^{b_N}
&\mathbb{Z}[\frac{1}{N}]/\mathbb{Z}\ar[r] &0}\label{tnse} \end{equation}
where $a_N$ is the inclusion sending $\zeta\in \mu_N$ to $(\zeta, 0) \in T[N]$, and $b_N$ sends $(X, \frac{i}{N})\in T[N]$ to $\frac{i}{N}$ mod $\mathbb{Z}$.

Then we recall a
smooth one-dimensional commutative group scheme $T$ over
$\mathbb{Z}[q^{\pm}]$, which is defined in \cite[Section 8.7]{KM85}. As a scheme, it is the disjoint union of schemes $T_{\alpha}$, indexed by $\alpha$'s running over all rational numbers in the interval $[0, 1)$, where each $T_{\alpha}$ is the scheme $\mathbb{G}_m$.

It sits in a short exact sequence of
group-schemes over $\mathbb{Z}[q^{\pm}]$:
$$0\longrightarrow \mathbb{G}_m\longrightarrow T\longrightarrow \mathbb{Q}/\mathbb{Z}\longrightarrow 0.$$

For any $\mathbb{Z}[q^{\pm}]$-algebra $R$ with connected spectrum, 
$$T(R)=\frac{R^*\times\mathbb{Q}}{\langle (q, -1)\rangle}$$ is still a group with the group structure defined in
\cite[(8.7.2.3), Section 8.7]{KM85}, i.e. for any $(X, \alpha)$, $(Y, \beta)$ in $T(R)$,
\begin{equation}
    (X, \alpha)\cdot (Y, \beta) := \begin{cases}(XY, \alpha+\beta) &\text{ if } \alpha+\beta<1;\\
    (\frac{XY}{q}, \alpha+\beta-1) &\text{  if  } \alpha+\beta \geq 1.
    \end{cases}
\end{equation}
Thus, $T$ is a functor from the category of
$\mathbb{Z}[q^{\pm}]$-algebras to the category of abelian groups.

Over a faithfully flat $\mathbb{Z}[q^{\pm}]$-algebra $R$ containing a compatible system of $N$'th roots $q^{\frac{1}{N}}$ of $q$ for every $N$, i.e. for every integer $N\geq 1$, we are given $Y_N\in R^*$ such that $Y_1=q$ and $(Y_{NM})^M= Y_N$ for every $M$, $N\geq 1$, $T$ is isomorphic to the product $$\mathbb{G}_m\times \mathbb{Q}/\mathbb{Z}$$ and the torsion points $T_{torsion}$ is therefore isomorphic to  
$$\mu_{\infty}\times \mathbb{Q}/\mathbb{Z}.$$

We have the conclusion below, which is Theorem 8.7.5 in
\cite{KM85}.

\begin{theorem}
There exists a faithfully flat $\mathbb{Z}[q^{\pm}]$-algebra $R$,
an elliptic curve $E/R$, and an isomorphism of ind-group-schemes
over $R$ $$T_{torsion}\otimes_{\mathbb{Z}[q^{\pm}]}
R\buildrel\sim\over\longrightarrow E_{tors},$$ such that for every
$N\geq 1$, the isomorphism on $N$-torsion points $T[N]\otimes
R\buildrel\sim\over\longrightarrow E[N]$ is compatible with
$e_N$-pairings. \label{torsionT}
\end{theorem}

Thus, we have the unique isomorphism of ind-group-schemes on
$\mathbb{Z}((q))$: $$T_{torsion}\otimes_{\mathbb{Z}[q^{\pm}]}
\mathbb{Z}((q))\buildrel\sim\over\longrightarrow Tate(q)_{tors}.$$
%The isomorphism is compatible with the canonical extension structure: for each $N\geq 1$, e

\begin{example}\label{keyisogtate}
As shown in \cite{36179} and \cite{AndoIsog}, over $\mathbb{Z}((q))$ there is an isogeny \begin{equation}[d]: Tate(q) \rightarrow Tate(q^d)\label{[d]tate} \end{equation} with kernel $\mu_d$.
And over $\mathbb{Z}((q^{\frac{1}{e}}))$ there is an isogeny 
\[\pi_e: Tate(q)\rightarrow Tate(q^{\frac{1}{e}}), %\quad z\mapsto z \mbox{ mod } q^{\frac{1}{e}}
\] with kernel $\mathbb{Z}[\frac{1}{e}]/\mathbb{Z}$. Especially, 
the restriction \[\pi_e: Tate(q^e) \rightarrow Tate(q)\] is the dual isogeny of $[e]: Tate(q) \rightarrow Tate(q^e)$.

Let $D(d, e)$ denote the $\mathbb{Z}((q))$-algebra
\begin{equation} D(d, e): = \mathbb{Z}((q))[q'_{d, e}]/\langle q^d- q'^e_{d,e}\rangle.  \label{defded}\end{equation}

There are two ring homomorphisms
\[ \mathbb{Z}((q)) \xrightarrow{i_{d, e}} D(d, e), \quad q\mapsto q\] \[ \mathbb{Z}((q)) \xrightarrow{j_{d, e}} D(d, e), \quad q\mapsto q'_{d, e}.\] From them we can define the group schemes $i^*_{d, e }Tate(q)$ and $ j^*_{d, e}Tate(q)$ as the pullbacks below.
\begin{equation} \label{ijdepbk}
\xymatrix{i^*_{d, e}Tate(q) \ar[r] \ar[d] & Tate(q) \ar[d] \\
\mbox{Spec}\; (D(d, e)) \ar[r]_{i_{d, e}^*} &\mbox{Spec} \; (\mathbb{Z}((q)))} \quad     \xymatrix{j^*_{d, e}Tate(q) \ar[r] \ar[d] & Tate(q) \ar[d] \\
\mbox{Spec}\; (D(d, e)) \ar[r]_{j_{d, e}^*} &\mbox{Spec} \; (\mathbb{Z}((q)))}  
\end{equation}

We can define an isogeny over $D(d, e)$ by the composition
\begin{equation}\phi'_{d, e}: i_{d, e}^*Tate(q) = Tate (q) \xrightarrow{[d]} Tate(q^d) \xrightarrow{\pi_e} Tate(q'_{d, e})=j^*_{d, e}Tate(q).\label{phiso}\end{equation}

Let $N=de$. The kernel of $\phi'_{d, e}$ is the $D(d, e)$-scheme
\begin{equation}T[d, e] := \{ z \in T[N]\otimes_{\mathbb{Z}[q^{\pm}]}\mathbb{Z}((q)) \; | \; z^d= q'^i_{d, e} \mbox{ for some } i.\} \label{kerphi'}\end{equation}
It fits into a short exact sequence 
\[0\rightarrow \mu_d \rightarrow  T[d, e] \rightarrow  \mathbb{Z}[\frac{1}{e}]/\mathbb{Z}\rightarrow 0.\]

And $T[d, e]$ is a finite subgroup of $Tate(q)$ of order $N$. 
\end{example}

\section{Tate K-theory}\label{qellshort}

\subsection{Quasi-elliptic cohomology and Tate K-theory} \label{qelldef}

In this section we give an explicit description of Tate K-theory via quasi-elliptic cohomology, which can be defined in terms of  equivariant K-theories. 
%In this section we recall the definition of quasi-elliptic cohomology in terms of equivariant K-theory and its relation to Tate K-theory. 
For more details on quasi-elliptic cohomology, please refer to \cite{HuanPower}  and \cite{Rez11}.

Let $X$ be a $G$-space. Let $G^{tors}\subseteq G$ be the set of
torsion elements of $G$. Let $\sigma\in G^{tors}$. And let \[C_G(\sigma)= \{g\in G \; | \;  \sigma g =g \sigma.\}\]
denote the centralizer.
The fixed point
space \[X^{\sigma}:= \{x\in X \; | \; x\cdot \sigma =x . \} \] is a $C_G(\sigma)$-space. Let
\[\Lambda_G(\sigma):= C_G(\sigma ) \times \mathbb{R}/\langle (\sigma, -1 )\rangle \] We can define a
$\Lambda_G(\sigma)$-action on $X^{\sigma}$ by
$$[g, t]\cdot x:=g\cdot x.$$
Then quasi-elliptic cohomology of the orbifold $X/\!\!/G$ is
defined by

\begin{definition}
\begin{equation}QEll^*(X/\!\!/G):=\prod_{\sigma\in
G^{tors}_{conj}}K^*_{\Lambda_G(\sigma)}(X^{\sigma})=\bigg(\prod_{\sigma\in
G^{tors}}K^*_{\Lambda_G(\sigma)}(X^{\sigma})\bigg)^G,\end{equation} where
$G^{tors}_{conj}$ is a set of representatives of $G$-conjugacy
classes in $G^{tors}$.
\end{definition}

We have the ring homomorphism
$$\mathbb{Z}[q^{\pm}]=K^0_{\mathbb{T}}(\mbox{pt})\buildrel{\pi^*}\over\longrightarrow K^0_{\Lambda_G(g)}(\mbox{pt})\longrightarrow
K^0_{\Lambda_G(g)}(X)$$ where $\pi: \Lambda_G(g)\longrightarrow
\mathbb{T}$ is the projection $[a, t]\mapsto e^{2\pi i t}$ and the
second is via the collapsing map $X\longrightarrow \mbox{pt}$. So
$QEll^*(X/\!\!/G)$ is naturally a
%$\mathbb{Z}[q^{\pm}]-$
$\mathbb{Z}[q^{\pm}]$-algebra. %why

\begin{proposition}The relation between quasi-elliptic cohomology and
Tate K-theory is
\begin{equation}QEll^*(X/\!\!/G)\otimes_{\mathbb{Z}[q^{\pm}]}\mathbb{Z}((q))=K^*_{Tate}(X/\!\!/G).
\label{tateqellequiv}\end{equation}\end{proposition}

\subsection{Moduli Problems} \label{tatemoduli}

In this section  We discuss several moduli problems for the Tate curve involving the classification of $N$-torsion points and that of  finite subgroups.
The computation of quasi-elliptic cohomology  via representation theory plays a role in the study of the moduli problems.

We have the computation \cite[Example 3.3]{Huansurvey} %via representation theory 
that \begin{equation}QEll(\mbox{pt}/\!\!/ (\mathbb{Z}/N\mathbb{Z}))= \prod\limits_{k=0}^{N-1}\mathbb{Z}[q^{\pm}][x_k]/(x_k^N-q^k),\label{tatenellc}\end{equation}
where each $x_k$ is the representation of $\Lambda_{\mathbb{Z}/N\mathbb{Z}}(k)$
defined by
\begin{equation}\begin{CD}\Lambda_{\mathbb{Z}/N\mathbb{Z}}(k)=(\mathbb{Z}\times\mathbb{R})/(\mathbb{Z}(N,0)+\mathbb{Z}(1,k))
@>{[a,t]\mapsto[(kt-a)/N]}>> \mathbb{R}/\mathbb{Z}=\mathbb{T}
@>{q}>> U(1).\end{CD}\label{xk}\end{equation} %As shown in Section \ref{tatecurve},
By \eqref{tindef} and \eqref{tatenellc}, we have the isomorphism below.
\begin{proposition}$T[N]\cong \mbox{Spec}\;(QEll(\mbox{pt}/\!\!/
(\mathbb{Z}/N\mathbb{Z}))).$ \label{NtorsQEll} \end{proposition}

In \cite[Section 4]{HuanPower}, we construct a power operation $\mathbb{P}_N$
for quasi-elliptic cohomology, which extends uniquely to the stringy power operation $P_N^{string}$ for Tate K-theory \cite[Definition
3.15]{Gan13}. Via $\mathbb{P}_N$ we show by computing representation rings the following conclusion, which is part of \cite[Theorem 7.4]{HuanPower}. \begin{proposition}
\begin{equation}QEll(\mbox{pt}/\!\!/\Sigma_N)/\mathcal{I}^{\Sigma_N}_{tr}\cong
\prod_{N=de}\mathbb{Z}[q^{\pm}][q'_{d, e}]/\langle q^d-q'^e_{d, e}
\rangle,\label{ellc}\end{equation} where \[q'= \prod_{N=de}q'_{d, e}\] is the image of $q$
under the power operation $\mathbb{P}_N$ and \begin{equation}\mathcal{I}^{\Sigma_N}_{tr}:= \sum_{\substack{i+j=N,\\
N>j>0}}\mbox{Image}[\mathcal{I}^{\Sigma_N}_{\Sigma_i\times\Sigma_j}:
QEll(\mbox{pt}/\!\!/\Sigma_i\times\Sigma_j)\longrightarrow
QEll(\mbox{pt}/\!\!/\Sigma_N)]\label{transferidealqec}\end{equation}
is the transfer ideal for quasi-elliptic cohomology. The product
goes over all the ordered pairs of positive integers $(d, e)$ such
that $N=de$. \label{tategplem}\end{proposition}

Applying the relation (\ref{tateqellequiv}), we can get the
conclusion below as a corollary of Proposition \ref{tategplem}, which is part of \cite[Theorem 7.4]{HuanPower}.
\begin{theorem}The Tate K-theory of symmetric groups modulo the
transfer ideal $I^{Tate}_{tr}$ classifies the finite subgroups of
the Tate curve. Explicitly,
\begin{equation}K_{Tate}(\mbox{pt}/\!\!/\Sigma_N)/I^{\Sigma_N}_{tr}\cong
\prod_{N=de} D(d, e)\label{tatec}\end{equation} 
where $D(d, e)$ is the ring defined in \eqref{defded}. And 
\[q':= \prod_{N=de}q'_{d, e}\] is the image of
$q$ under the power operation $P^{string}_N$ constructed in  \cite[Definition
3.15]{Gan13} and \begin{equation}I^{\Sigma_N}_{tr}:= \sum_{\substack{i+j=N,\\
N>j>0}}\mbox{Image}[I^{\Sigma_N}_{\Sigma_i\times\Sigma_j}:
K_{Tate}(\mbox{pt}/\!\!/\Sigma_i\times\Sigma_j)\longrightarrow
K_{Tate}(\mbox{pt}/\!\!/\Sigma_N)]\label{transferidealtatek}\end{equation}
is the transfer ideal of Tate K-theory. The product goes over
all the ordered pairs of positive integers $(d, e)$ such that
$N=de$.\label{tatefiniteclass} \end{theorem}

\subsection{Power operation }

Moreover, via the power operation $\mathbb{P}_N$ we will construct a
new operation $$\overline{P}_N: QEll(X/\!\!/G) \longrightarrow
QEll(X/\!\!/G)\otimes_{\mathbb{Z}[q^{\pm}]}
QEll(\mbox{pt}/\!\!/{\Sigma_N})/\mathcal{I}^{\Sigma_N}_{tr}$$ of
quasi-elliptic cohomology. It is essential in the construction of
the universal finite subgroup of order $N$ of $Tate(q)$.
\begin{proposition} \label{addpower}
The composition
\begin{align*}\overline{P}_N: &QEll(X/\!\!/G)
\buildrel{\mathbb{P}_N}\over\longrightarrow
QEll(X^{\times N}/\!\!/{G\wr\Sigma_N})
\buildrel{res}\over\longrightarrow QEll(X^{\times N}/\!\!/{G\times
\Sigma_N})\\  &\buildrel{diag^*}\over\longrightarrow
QEll(X/\!\!/{G\times\Sigma_N}) \cong QEll(X/\!\!/G)\otimes_{\mathbb{Z}[q^{\pm}]} QEll(\mbox{pt}/\!\!/{\Sigma_N}) \\
  &\longrightarrow QEll(X/\!\!/G)\otimes_{\mathbb{Z}[q^{\pm}]} QEll(\mbox{pt}/\!\!/{\Sigma_N})/\mathcal{I}^{\Sigma_N}_{tr} \\ &\cong
QEll(X/\!\!/G)\otimes_{\mathbb{Z}[q^{\pm}]}\prod_{N=de}\mathbb{Z}[q^{\pm}][q'_{d, e} ]/\langle q^d-q'^e_{d, e}
\rangle\end{align*}defines a ring homomorphism, where $res$ is the
restriction map given by the inclusion
$$G\times\Sigma_N\hookrightarrow G\wr\Sigma_N\mbox{,    } (g, \sigma)\mapsto (g, \cdots g;
\sigma),$$ and $diag$ is the diagonal map $$X\longrightarrow X^{\times
N}\mbox{,    } x\mapsto (x, \cdots x).$$ In addition, we refer the reader to \cite[(4.1), (4.2)]{HuanPower} for the construction of the power operation
 $\{\overline{P}_N\}_N$ and \cite[(4.17)]{HuanPower} for the explicit formula of it. 
\end{proposition}

The operation $\overline{P}_N$ sends $q$ to $q'$. In addition, it extends uniquely to a ring homomorphism 
\[\overline{P^{string}}_N: K_{Tate}(X/\!\!/G) \longrightarrow K_{Tate}(X/\!\!/G)\otimes_{\mathbb{Z}((q))}\prod_{N=de}D(d, e)\cong 
\prod_{N=de}D(d, e)\otimes_{i_{d, e}} K_{Tate}(X/\!\!/G)\] constructed in \cite[Section 5.4]{Gan07}.
%For more details, please refer to \cite{HuanPower}  and \cite{Rez11}.
The operations gives an isogeny of formal groups
\begin{equation}
    \label{linepfgp}
    \overline{P^{string}}^*_N:  \coprod_{N=de}   i_{d, e}^*Tate(q)  \rightarrow Tate(q)
\end{equation}

\section{The universal finite subgroup of  the Tate curve}\label{mainpart}

In this section we construct a  finite subgroup over $D_N$ of order $N$ of $i_N^*Tate(q)$ as the kernel of an isogeny $\psi$, 
%\[G_{univ}:=(D_N, \mbox{Ker} (\psi)< i_N^*Tate(q)[N])\] of of order $N$ of $Tate(q)$.    % over the ring $\prod\limits_{N=de} D(d, e)$. 
which is defined in \eqref{formulapsio}. An explicit description of $\mbox{Ker} (\psi)$ is given in \eqref{univformula}.
We prove in Section \ref{maincc} that $\mbox{Ker} (\psi)$ is the universal finite subgroup of $Tate(q)$ of order $N$.

%In Section \ref{state}, \ref{ttate} and \ref{mapsi} we will give the construction of $G_{univ}:=(O_{Sub_{\Sigma_N}}, \mbox{Ker} (\psi) < s^*Tate[N])$. In Section \ref{maincc} we  give an explicit description of $G_{univ}$ and prove the main conclusion that $G_{univ}$ is universal.

Let $D_N$  denote the product 
\[\prod_{N=de} D(d, e).\] And let  $i_N$ denote the map
\[i_N:=\prod_{N=de} i_{d, e}:  \mathbb{Z}((q)) \rightarrow D_N, \quad q\mapsto \prod_{N=de}q.\] 
And let $j_N$ denote the map \[j_N:=\prod_{N=de} j_{d, e}:  \mathbb{Z}((q)) \rightarrow D_N, \quad q\mapsto \prod_{N=de} q'_{d, e}.\]
We have the isomorphisms 
\begin{align*}
i_{N}^*Tate(q) &\cong \coprod_{N=de}   i_{d, e}^*Tate(q) ;\\
j_{N}^*Tate(q) &\cong \coprod_{N=de}   j_{d, e}^*Tate(q) ,
\end{align*}
and the isogeny over $D_N$
\begin{equation}\phi'_N:= \coprod_{N=de}\phi'_{d, e}: i_{N}^*Tate(q) \longrightarrow j^*_{N}Tate(q)\label{phin}\end{equation}
where each $\phi'_{d, e}$ is defined in \eqref{phiso}.

The kernel of $\phi'_{N}$ is the $D_N$-scheme
\[\coprod_{N=de}T[d, e]. \]

From the pullback squares \eqref{ijdepbk} we obtain the pullback squares below.
\begin{equation} \label{ijnpbkn}
\xymatrix{i^*_{N}Tate(q) \ar[r] \ar[d] & Tate(q) \ar[d] \\
\mbox{Spec}\; (D_N) \ar[r]_{i_{N}^*} &\mbox{Spec} \; (\mathbb{Z}((q)))} \quad     \xymatrix{j^*_{N}Tate(q) \ar[r] \ar[d] & Tate(q) \ar[d] \\
\mbox{Spec}\; (D_N) \ar[r]_{j_{N}^*} &\mbox{Spec} \; (\mathbb{Z}((q)))}  
\end{equation}

To study the finite subgroups of $Tate(q)$ of order $N$, it suffices to work inside the world of $N$-torsion points of $Tate(q)$. 
%By Theorem\ref{torsionT}, we can work on the group scheme $T$ over $\mathbb{Z}[q^{\pm}]$ and study the finite subgroups of $T$, which has a one-to-one correspondence with the finite subgroups of $Tate(q)$.
By Theorem  \ref{torsionT}, the $N$-torsion points $Tate(q)[N]$ of $Tate(q)$ is isomorphic to 
\begin{equation}Tate(q)[N]\cong T[N]\otimes_{\mathbb{Z}[q^{\pm}]} \mathbb{Z}((q)) \cong \coprod_{i=0}^{N-1} \mbox{Spec} \; (\mathbb{Z}((q))[x_k]/\langle x_k^N-q^i \rangle).\label{tateqn}\end{equation}
By Proposition \ref{tateqellequiv} and Proposition \ref{NtorsQEll}, we have the isomorphism of rings of functions
\begin{equation}
    \label{otateqn}
    O_{Tate(q)[N]}\cong K^*_{Tate}(\mbox{pt}/\!\!/(\mathbb{Z}/N\mathbb{Z})).
\end{equation}

 The $N$-torsion points of $i^*_{N}Tate(q)$ is \[(i^*_{N}Tate(q))[N] = i^*_{N}(Tate(q)[N]).\] And 
the $N$-torsion points of $j^*_{N}Tate(q)$ is \[(j^*_{N}Tate(q))[N] = j^*_{N}(Tate(q)[N]).\]
Thus, we have the pullback squares below.
\begin{equation} \label{ijnpbkn}
\xymatrix{i^*_{N}Tate(q)[N] \ar[r] \ar[d] & Tate(q)[N] \ar[d] \\
\mbox{Spec}\; (D_N) \ar[r]_{i_{N}^*} &\mbox{Spec} \; (\mathbb{Z}((q)))} \quad     \xymatrix{j^*_{N}Tate(q)[N] \ar[r] \ar[d] & Tate(q)[N] \ar[d] \\
\mbox{Spec}\; (D_N) \ar[r]_{j_{N}^*} &\mbox{Spec} \; (\mathbb{Z}((q)))}  
\end{equation}
And we have the pushout squares of the induced maps on the rings of functions
\begin{equation}
    \label{ijnpbkin}
\xymatrix{    \mathbb{Z}((q)) \ar[r]^{i_N} \ar[d] & D_N
    \ar[d] \\  K_{Tate}(\mbox{pt}/\!\!/(\mathbb{Z}/N\mathbb{Z})) \ar[r]
    &  D_N\otimes_{i_N} K_{Tate}(\mbox{pt}/\!\!/(\mathbb{Z}/N\mathbb{Z}))}
\end{equation}
\begin{equation}
    \label{ijnpbkjn}
\xymatrix{    \mathbb{Z}((q)) \ar[r]^{j_N} \ar[d] & D_N
    \ar[d] \\  K_{Tate}(\mbox{pt}/\!\!/(\mathbb{Z}/N\mathbb{Z})) \ar[r]
    &  D_N\otimes_{j_N} K_{Tate}(\mbox{pt}/\!\!/(\mathbb{Z}/N\mathbb{Z}))}
\end{equation}

By the universal property of the pushout, there is a unique map
\begin{equation}\psi^*: D_N\otimes_{j_N} K_{Tate}(\mbox{pt}/\!\!/(\mathbb{Z}/N\mathbb{Z}))\longrightarrow
D_N\otimes_{i_N} K_{Tate}(\mbox{pt}/\!\!/(\mathbb{Z}/N\mathbb{Z}))\label{isoznop}\end{equation} making the diagrams below commute. \begin{equation}
\xymatrix{\mathbb{Z}((q))\ar[r]^{j_N}\ar[d] &D_N \ar[d]\ar@/^/[rdd]^{p} &\\  K_{Tate}(\mbox{pt}/\!\!/(\mathbb{Z}/N\mathbb{Z})) \ar[r]\ar@/_/[rrd]^{\overline{P^{string}}_N} &
 D_N\otimes_{j_N} K_{Tate}(\mbox{pt}/\!\!/(\mathbb{Z}/N\mathbb{Z})) \ar@{.>}[dr]|{\psi^*} &\\ & &
D_N\otimes_{i_N} K_{Tate}(\mbox{pt}/\!\!/(\mathbb{Z}/N\mathbb{Z})).}\label{psidef}\end{equation}

We apply Spec to the diagram \eqref{psidef} and obtain the commutative diagrams of group schemes below. \begin{equation}\xymatrix{i^*_NTate(q)[N]\ar[rd]|{\psi} \ar@/^/[rrd]^{\overline{P^{string}}_N^*}\ar@/_/[rdd] \\ &j_N^*Tate(q)[N]\ar[r]\ar[d] &Tate(q)[N] \ar[d] \\
&\mbox{Spec}\; D_N \ar[r]_{j^*_N} &\mbox{Spec}\;(\mathbb{Z}((q))) }\label{psidefinv}\end{equation}

Next we show the explicit formula for $\psi^*$.

We define an element \[x_k\in K_{Tate}(\mbox{pt}/\!\!/(\mathbb{Z}/N\mathbb{Z})) \] in (\ref{xk}). For
any $\mathbb{Z}((q))$-algebra $R$ with connected spectrum, the formula of the map \[x_k:
Tate(q)[N](R)\longrightarrow R\] is $$x_k([a, t])=\begin{cases}a,
&\text{if $t=\frac{k}{N}$ with $k=0, 1, \cdots N-1$;}\\ 0,
&\text{if $[a, t]\neq [a', \frac{k}{N}]$ for any
$a'$.}\end{cases}$$ Note that, in $Tate(q)[N](R)$, $[a, t+1]= [aq, t]$.

By the formula for the operation 
$$\overline{P^{string}}_N: K_{Tate}(\mbox{pt}/\!\!/(\mathbb{Z}/N\mathbb{Z}))\longrightarrow  D_N\otimes_{i_N} K_{Tate}(\mbox{pt}/\!\!/(\mathbb{Z}/N\mathbb{Z}))$$
 given in Proposition \ref{addpower}, it sends
$\prod\limits_{m=0}^{N-1}q$ to \[\prod\limits_{m=0}^{N-1}
q'_{d, e}\otimes 1= q',\] and sends each $x_k$ to
\[\prod\limits_{\substack{N=de\\
e|k}}\prod\limits_{\alpha_m=0}^{e-1} {q'_{d, e}}^{-\alpha_m}\otimes x_m^d\]
where $m=\frac{k}{e}+\alpha_md$.

Therefore, %the map $\psi^*$ sends both $q\otimes 1$ and $1\otimes q'$ to $1\otimes q'$, %$x_k\otimes 1$ to $x_k\otimes 1$, and $1\otimes q$ to $ q\otimes 1$. \comment{What is $\psi^*$?}
by the commutativity of the diagram \eqref{psidef},  
$\psi^*$ sends $1\otimes (\prod\limits_{m=0}^{N-1}q)$ and  $ q' \otimes 1$ to the image of \[\prod\limits_{m=0}^{N-1}q\in   K_{Tate}(\mbox{pt}/\!\!/(\mathbb{Z}/N\mathbb{Z}))\]  under $\overline{P^{string}}_N$, i.e.
\[\prod\limits_{m=0}^{N-1}
q'_{d, e}\otimes 1= q',\]  and it sends $x_k\otimes 1$ to the image of $x_k$ under $\overline{P^{string}}_N$, i.e. 
\[\prod\limits_{\substack{N=de\\
e|k}}\prod\limits_{\alpha_m=0}^{e-1} {q'_{d, e}}^{-\alpha_m}\otimes x_m^d\]
with each $m=\frac{k}{e}+\alpha_md$, 
and $\psi^*$ sends 
 $ q  \otimes 1$ to the image of  $ q\in D_N$ under $p$, i.e. 
 \[ q \otimes 1\in  D_N\otimes_{i_N} K_{Tate}(\mbox{pt}/\!\!/(\mathbb{Z}/N\mathbb{Z})).\]

\bigskip

Next, we give an explicit formula for $\psi$.

For any $D(d, e)$-algebra $ A$ with connected spectrum, we have the equivalences for the $N$-torsion points
\emptycomment{take $N$-torsion points on the RHS}
\begin{equation}i^*_{d, e} Tate(q)[N](A)=
\frac{A^*\times (\mathbb{Z}/N\mathbb{Z})}{\langle (q, -1)\rangle}\mbox{ }\mbox{
and } \mbox{    }  j^*_{d, e} Tate(q)[N](A)= \frac{A^*\times
(\mathbb{Z}/N\mathbb{Z})}{\langle (q'_{d, e}, -1)\rangle}. \label{stateA}\end{equation} Under this identification, the map
\[\psi_{d, e}(A): i^*_{d, e} Tate(q)[N](A)\longrightarrow j^*_{d, e} Tate(q)[N](A)\] is defined by 
\begin{equation} [a, x]\mapsto [a^d, ex].\label{psifor}\end{equation} It is well-defined since $[q^d,
-e]= [q'^e, -e]=0$ in $j^*_{d, e} Tate(q)[N](A)$.

The map \begin{equation}\psi: i^*Tate(q)[N]\longrightarrow j^*Tate(q)[N]\label{formulapsio}\end{equation} defined in
the diagram (\ref{psidef}) can be constructed as the coproduct of
the maps
$$\psi_{d, e}: i^*_{d, e} Tate(q)[N]\longrightarrow j^*_{d, e}Tate(q)[N].$$ For each $(d, e)$, 
\[(\psi^*_{d,
e}x_k)[a, x]=x_k(\psi_{d, e}[a, x])=x_k([a^d, ex]).\]
Note that $\psi^*_{d, e}x_k$ is in $  D_N\otimes_{i_N} K_{Tate}(\mbox{pt}/\!\!/(\mathbb{Z}/N\mathbb{Z}))$.

With the formula given in \eqref{psifor}, we have  %we can see $G_{univ}$ is the kernel of the map $\psi$ and
an explicit formula for $(\mbox{Ker}(\psi))(A)$.
\begin{align}(\mbox{Ker}(\psi))(A) &=\!\prod\limits_{N=de}\! \mbox{Ker} (\psi_{d, e}) (A)\\
&=
\!\prod\limits_{N=de}\!\{[a ,x]\!\in\! i^*_{d,
e}Tate(q)[N](A)\mbox{ }|\mbox{ }[a^d, ex]=[1,
0]\}.\label{univformula} \end{align} 
%for any $D(d, e)-$algebra $A$ with connected spectrum.

$\mbox{Ker}(\psi)$ is isomorphic to \[%\coprod_{N=de}\mbox{Ker} (g_{d, e})  =
\mbox{Ker}(\psi'_N)=\coprod_{N=de} T[d, e]
\]
where each $T[d, e]$    %, defined in \eqref{kerphi'},
 is the kernel of the isogeny \[\phi'_{d, e}: i^*_{d, e}Tate(q) \rightarrow j^*_{d, e}Tate(q)\] defined in \eqref{phiso}.

\begin{remark}
The isogeny $\psi$ is the same as the restriction of the isogeny \[\phi'_N: i_{N}^*Tate(q) \longrightarrow j^*_{N}Tate(q)\] defined in \eqref{phin}
to the $N$-torsion part $i^*_{N} Tate(q)[N]$.
By \cite[Theorem A, Corollary 6.7]{AndoIsog}, there is a ring operation
\[ \Psi_N: D_N\otimes_{i_N} K_{Tate}(-) \longrightarrow  D_N\otimes_{j_N} K_{Tate}( - )   \] such that 
$\Psi_N(\mathbb{C}P^{\infty}) = {\phi'}_N^*$. And $\Psi_N(B(\mathbb{Z}/N\mathbb{Z}))$ is the homomorphism $\psi^*$ defined in \eqref{isoznop}.
\end{remark}

\bigskip
\section{The main theorem}\label{maincc}

Now we are ready to state the main conclusion of this paper.

In this section, by a $\mathbb{Z}((q))$-algebra $R$, we mean an faithfully flat $\mathbb{Z}((q))$-algebra $R$ containing a compatible system of $N$'th roots $q^{\frac{1}{N}}$ of $q$ for every $N$, i.e. for every integer $N\geq 1$, we are given $Y_N\in R^*$ such that $Y_1=q$ and $(Y_{NM})^M= Y_N$ for every $M$, $N\geq 1$.

\begin{theorem} The  universal finite subgroup of order $N$ of $Tate(q)$ is the pair $$G_{univ}:=(D_N, \mbox{Ker} (\psi)< i_N^*Tate(q)[N])$$
in the sense that for any
$\mathbb{Z}((q))$-algebra $R$, there is a 1-1 correspondence which is natural
\begin{equation}\xymatrix{\{\mathbb{Z}((q))-\mbox{algebra maps }D_N\longrightarrow R
\}\ar@{^{<}-_{>}}[d]\\  \{\mbox{finite subgroup schemes }
G\leqslant Tate(q)[N]_R \mbox{   of order }N\}}\label{11main}\end{equation} where $Tate(q)[N]_R$ is the
pullback
$$\xymatrix{Tate(q)[N]_R\ar[r]\ar[d] &Tate(q)[N]\ar[d] \\ \mbox{Spec}\; (R)\ar[r]
&\mbox{Spec}\;(\mathbb{Z}((q )) )}.$$ It is a group scheme over
$\mbox{Spec}\;( R)$. In other words, given a subgroup scheme
$G\leqslant Tate(q)[N]_R$ of degree $N$, there exists a unique pullback
square
\begin{equation}\xymatrix{G \ar@{^{(}->}[r]\ar@{^{(}->}[d] &\mbox{Ker} (\psi) \ar@{^{(}->}[d] \\ Tate(q)[N]_R\ar[r]\ar[d] &i_N^*Tate(q)[N] \ar[d] \\
\mbox{Spec} \; (R)\ar[r] &\mbox{Spec}\; (D_N)
.}\label{universaldef}\end{equation} \label{mainthm}\end{theorem}

\begin{remark}
Given the subgroup $G\leqslant Tate(q)[N]_R$ in the construction, there is  an isogeny \[\psi_R: i^*_NTate(q)[N]_R\rightarrow j_N^*Tate(q)[N]_R\] over $R$ such that $G$ is a subgroup of the kernel of $\psi_R$, where $j^*_NTate(q)[N]_R$ and  $i^*_NTate(q)[N]_R$  are the pullbacks
\begin{equation}\xymatrix{j_N^*Tate(q)[N]_R\ar[r]\ar[d] &j_N^*Tate(q)[N] \ar[d] \\
\mbox{Spec} \; (R)\ar[r] &\mbox{Spec}\; (D_N)
} \quad \xymatrix{i_N^*Tate(q)[N]_R\ar[r]\ar[d] &i_N^*Tate(q)[N] \ar[d] \\
\mbox{Spec} \; (R)\ar[r] &\mbox{Spec}\; (D_N)
,}\label{universaldefj}\end{equation} 
and $\psi_R$ is constructed as the base change of the isogeny \[\psi: i^*_NTate(q)[N]\rightarrow j^*_NTate(q)[N].\]
\end{remark}

\begin{proof}[Proof of Theorem \ref{mainthm}]

We prove the conclusion by three steps.

\textbf{Step I:} 
Since $\mbox{Spec}\;(R)$ is connected, the image of each map 
\[\mbox{Spec}\;(R) \rightarrow \mbox{Spec}\; (D_N)= \coprod_{N=de} \mbox{Spec}\; (D(d, e))\]
lies in some component $\mbox{Spec}\; (D(d, e))$ of $ \mbox{Spec}\; (D_N)$.

We show, given a finite subgroup $G<Tate(q)[N]_R$ of
order $N$,  we can construct a map \[F^*_{d, e}: \mbox{Spec} \; (R)\longrightarrow \mbox{Spec}
\;
(D(d,e ))\] for some pair of integers $(d, e)$ with $N=de$.

For the subgroup $G<Tate(q)[N]_R$, we have the exact sequences of group schemes over $R$ and the commutative diagrams 

\[\xymatrix{0\ar[r] &\mathbb{G}_m\ar[r]
&Tate(q)[N]_R\ar[r] &\mathbb{Q}/\mathbb{Z}\ar[r] &0 \\ %where $G_m$ is the multiplicative formal group.
 0\ar[r] &\mathbb{G}_m\cap G \ar[r] \ar@{^{(}->}[u] &G\ar[r] \ar@{^{(}->}[u]
&\mathbb{Z}[\frac{1}{e}]/\mathbb{Z}\ar[r] \ar@{^{(}->}[u] &0}\] for some positive integer $e \geq 1$ dividing $N$. %where   $\mu_d$
\emptycomment{The subgroup has order N as a subgroup scheme but the set of R-points of H does not. It can be
larger if R is disconnected. It can be smaller if some of these points are not defined over R, such as the points of $\mu_d$ if R does not have the roots of unity. It may always be smaller if R has positive characteristic p that divides N. The quotient group
$H/R^*\cap H$ is contained inside the R-points of $Q/Z$.}
The group $\mathbb{G}_m\cap G$ is the kernel of the projection $G\longrightarrow \mathbb{Q}/\mathbb{Z}$.
Let $d$ denote the order of $\mathbb{G}_m\cap G$. We have $N=de$.

\emptycomment{
Let $d=|R^*\cap H|$ and
$e=|H/(R^*\cap H)|$. Then $H/(R^*\cap
H)=\mathbb{Z}[\frac{1}{e}]/\mathbb{Z}\cong
\mathbb{Z}/e\mathbb{Z}$. %Note $N=de$.
The group $R^*\cap H$ is the kernel of the projection $$
H\longrightarrow H/(R^*\cap H)\mbox{,   }[x, \frac{1}{e}]\mapsto
\frac{1}{e}.$$}

In addition, we have the exact sequences of group schemes over $R$ 
\begin{equation}\xymatrix{0\ar[r] &\mu_d \ar@{=}[d]\ar[r] & G\ar[r]\ar@{^{(}->}[d] &\mathbb{Z}[\frac{1}{e}]/\mathbb{Z}
\ar@{.>}[d]^{r}_{\exists !}\ar[r] &0 \\0\ar[r] &\mu_d
\ar[r] & Tate(q)_R\ar[r]^{[d]}&Tate(q^d)_R \ar[r] &0
}\label{diagqr'}\end{equation} where $Tate(q)_R$ is the pullback
$$\xymatrix{Tate(q)_R\ar[r]\ar[d] &Tate(q)\ar[d] \\ \mbox{Spec}\; (R)\ar[r]
&\mbox{Spec}  \; (\mathbb{Z}((q)))}$$ and $[d]_R$ is the isogeny over $R$ %defined in \eqref{[d]tate} with 
\[[d]_R([x, \lambda])= [x^d, \lambda].\]
There exists a unique homomorphism \[r: \mathbb{Z}[\frac{1}{e}]/\mathbb{Z} \rightarrow Tate(q^d)_R\] making the diagram
\eqref{diagqr'} commute.

Let $a\in\mathbb{Z}[\frac{1}{ e}]/\mathbb{Z}$ be a
generator. \begin{equation}r(a)=[ q'_R, \frac{1}{e}]\label{q'def}\end{equation} for some $q'_R\in R^*$. Note that, in $Tate(q^d)_R $,  $0=e \cdot [q'_R, \frac{1}{e}]= [q'^e_R, 1]=
[q'^e_R q^{-d}, 0]$ and \[q'^e_R= q^d.\] Then we define a ring map $$F_{d, e}: D(d, e) \longrightarrow R$$ by sending $q$ to $q$ and sending $q'_{d, e}$ to $q'_R$.

Then the map $F^*_{d, e}$ defined by the composition \[ \mbox{Spec} \;( R)\longrightarrow\mbox{Spec} \; ( D(d, e) ) \hookrightarrow \mbox{Spec}
\;(D_N)\] is the one we want.

\bigskip

\textbf{Step II:} We show that, given a $\mathbb{Z}((q))$-algebra map
$D_N\buildrel{g}\over\longrightarrow R$, we can
construct a finite subgroup $G < Tate(q) [N]_R$ of order $N$.
%, via the diagram (\ref{universaldef}).

Since $Tate(q)[N]$ is a finite flat group scheme over $\mathbb{Z}((q))$  for any integer $N\geq 1$,  and $R$ is an faithful flat $\mathbb{Z}((q))$-algebra, thus             %, by Theorem \ref{torsionT}, 
the group scheme $Tate(q)[N]_R$ is finite flat. %By the definition of finite flat  group scheme, t
The order of it is defined by the locally constant map 
 \begin{align*}
     \delta:
\mbox{Spec} \; (R) & \longrightarrow \mathbb{Z}_{\geq 0} \\
p &\mapsto \mbox{rank}_{R_p} (O_{Tate(q)[N]_R})_p
 \end{align*}
 
Then there exists a $\mathbb{Z}((q ))$-algebras $R_{d, e}$ for each pair of  integers $(d, e)$ with $N= de$ such that 
\[\mbox{Spec} \; (R_{d, e}) = \{x\in \mbox{Spec} \; (R)\mbox{   }|\mbox{  }\delta(x)= d\}.\] And
$R$ is the product
\[\prod\limits_{N=de} R_{d, e}.\] 
%%What is R_{d, e}?

Thus the ring homomorphism \[\alpha: D_N \longrightarrow R\] is  the product of the 
maps $$\prod_{N=de}(\alpha_{d, e}: D(d, e)\longrightarrow R_{d, e}).$$

If there is no element $x$ in $R^*$ such that
$x^e=q^d$, then $\alpha_{d, e}=0$. In this case the contribution of $\alpha_{d, e}$ to
the subgroup $G$ is $0$. There is only one factor $\alpha_{d_0, e_0}$ that is
non-trivial.  We have the pull-back diagrams
$$\xymatrix{G\ar[d]\ar[r] &\mbox{Ker}(\psi_{d_0, e_0}) \ar[d]
\\ Tate(q)[N]_{R_{d_0, e_0}}\ar[d]\ar[r] &i^*_{d_0, e_0} Tate(q)[N] \ar[d] \\ \mbox{Spec} \; (R_{d_0,
e_0})\ar[r]^{\alpha_{d_0, e_0}^*} &\mbox{Spec} \; (D(d_0, e_0))}$$

And, for any $D(d, e)-$algebra $A$ with connected spectrum, $G(A)$ is the factor
\[\{[t, x]\in i^*_{d, e}Tate(q)[N](A)\mbox{ }|\mbox{ }[t^d,
ex]=[1, 0]\}.\] Thus, the order of $G$ is $N$.
\emptycomment{ Should x in R?}
\bigskip

\textbf{Step III:} In this step we check the two maps in
Step I and Step II are the inverse of each other. Then, we have the 1-1 correspondence (\ref{11main}).

Let  $H <Tate(q)[N]_R$ denote a finite group of
order $N$. %As discussed in Example \ref{finited}, it is determined by a pair $(d, e)$ of positive integers with $N=de$. 
By Step I, we have a $\mathbb{Z}((q))$-algebra map $$F_{d, e}: D(d, e) \longrightarrow R$$
for some pair $(d, e)$ of positive integers with $N=de$. %This gives the map $g_{d, e}$ in Step II. 
By the construction of $F_{d, e}$ in Step I,  $H$ is the pullback in the diagram 
$$\xymatrix{H\ar[d]\ar[r] &\mbox{Ker}(\psi_{d, e}) \ar[d]
\\ Tate(q)[N]_{R_{d, e}}\ar[d]\ar[r] &i^*_{d, e} Tate(q)[N] \ar[d] \\ \mbox{Spec} \; (R_{d,
e})\ar[r]^{F_{d, e}^*} &\mbox{Spec} \; (D(d, e) )}$$
By the uniqueness of pullback, we know the subgroup $G$ of $Tate(q)[N]_R$ obtained from $F_{d, e}:  D(d, e) \longrightarrow R$ in Step II is isomorphic to $H$.

For the other direction, given a $\mathbb{Z}((q))$-algebra map
$D_N\buildrel{\alpha}\over\longrightarrow R$, we can
construct a finite subgroup $G<Tate(q)[N]_R$ as the pullback of the diagram
$$\xymatrix{G\ar[d]\ar[r] &\mbox{Ker}(\psi_{d, e}) \ar[d]
\\ Tate(q)[N]_{R_{d, e}}\ar[d]\ar[r] &i^*_{d, e} Tate[N] \ar[d] \\ \mbox{Spec} \; (R_{d,
e})\ar[r]^{\alpha_{d, e}^*} &\mbox{Spec} \; (D(d, e))}$$
Especially, there is only one pair $(d, e)$ that is relevant to the construction of $G$. Then, the construction in Step I recovers the map $\alpha$.

\emptycomment{ Why are they inverse to each other?}
\end{proof}

%\begin{remark}    \end{remark}

\newpage
\bibliographystyle{amsplain}


\end{document}